\documentclass[twoside]{article}

\usepackage{a4wide, amsmath,  amssymb, amscd, mathptmx, amsthm}
\usepackage{amsfonts}
\usepackage{stmaryrd}
\usepackage{bbm}

\usepackage{color}

\usepackage{pstricks}
\usepackage[all]{xy}\CompileMatrices\SelectTips{cm}{12}

\theoremstyle{plain}
\newtheorem{Thm}{\sc Theorem}[section]
\newtheorem{theorem}[Thm]{\sc Theorem}
\newtheorem{corollary}[Thm]{\sc Corollary}
\newtheorem*{corollary*}{\sc Corollary}
\newtheorem{proposition}[Thm]{\sc Proposition}
\newtheorem*{proposition*}{\sc Proposition}
\newtheorem{lemma}[Thm]{\sc Lemma}

\theoremstyle{remark}
\newtheorem{remark}[Thm]{Remark}
\newtheorem{example}[Thm]{Example}
\newtheorem*{example*}{Example}
\newtheorem*{remark*}{Remark}

\newcommand{\cC}{{\mathcal C}}

\newcommand{\cO}{{\mathcal O}}

\newcommand{\NN}{{\mathbb N}}

\newcommand{\PP}{{\mathbb P}}
\newcommand{\QQ}{{\mathbb Q}}

\newcommand{\Flag}{{\mathop{\rm Flag}}}

\newcommand{\Gras}{\mathop{\rm Gr}}

\newcommand{\id}{{\mathop{\rm id}}}

\newcommand{\Spec}{\mathop{\rm Spec \, }}

\newcommand{\Gr}{\mathop{{\rm Gr}}}

\newcommand{\ch}{{\mathop{\rm ch \, }}}

\newcommand{\Sch}{{\mathop{{\rm Sch }}}}

\newcommand{\mult}{{\mathop{{\rm mult\, }}}}

\newcommand{\nf}{\mathop{{\rm nf}}}
\newcommand{\redu}{\mathop{{\rm red}}}

\newcommand{\GL}{\mathop{\rm GL}}

\newcommand{\Pic}{{\mathop{\rm Pic\, }}}

\newcommand{\Vect}{{\mathop{\rm Vect }}}

\newcommand{\Sets}{{\mathop{{\rm Sets}}}}

\begin{document}

\markboth{\rm }{\rm  }

\title{Approximation of semistable bundles on smooth algebraic varieties}
\author{Adrian Langer}

\date{\today}

\maketitle

{\sc Address:}\\
Institute of Mathematics, University of Warsaw,
ul.\ Banacha 2, 02-097 Warszawa, Poland\\
e-mail: {\tt alan@mimuw.edu.pl}

\medskip

\begin{abstract} 
	We prove some strong results on approximation of strongly semistable bundles with vanishing numerical 
	Chern classes by filtrations, whose quotients are line bundles of similar slope. This generalizes some earlier results of  Parameswaran--Subramanian in the curve case and Koley--Parameswaran in the surface case and it confirms the conjecture posed by Koley and Parameswaran.
\end{abstract}

\medskip

\thanks{
\emph{To the memory of Professor Piotr Pragacz}}

\section*{Introduction}

Deligne and Sullivan proved in \cite{DS} that a flat vector bundle on a complete smooth complex variety admits an \'etale cover such that its pull back is trivial as a  $\mathcal C^{\infty}$-vector bundle. This is obviously false in the algebraic case even for flat line bundles and even if we allow arbitrary finite surjective morphisms. Instead, one can try to understand how far the pull-backs of a given flat vector bundle are from the trivial bundle in the case of finite surjective morphisms. The best one can hope for is that  for a given flat vector bundle $E$ on a smooth projective variety $X$ over some algebraically closed  field $k$ there exists 
a finite surjective morphism $f: Y\to X$ such that $f^*E$ is algebraically equivalent (i.e., it can be deformed) to the trivial flat vector bundle. This is indeed true for numerically flat vector bundles in case $k$ is an algebraic closure of a finite field (see \cite[Theorem 2.9]{FL}). However, the most interesting case, that of characteristic zero, seems out of reach. In this case the above prediction would imply a weak version of Bloch's conjecture on Chern classes of flat algebraic vector bundles: Chern classes of flat vector bundles should be, modulo torsion, algebraically equivalent to $0$.

One can also try to approximate vector bundles using finite surjective morphisms looking at sub-bundles of pull-backs of a given vector bundle. In the case of curves, such a study was done by Parameswaran and Subramanian in \cite{PS}, and they proved that every semistable rank $r$ vector bundle $E$ on a smooth projective curve $X$ admits a sequence of finite surjective morphisms $f_m:Y_m \to X$ such that the maximal possible slopes of subbundles of rank $1\le s<r$ in pull-backs of $E$, divided by the degree of $f_m$, converge to the slope of $E$.

A generalization of this fact to the surface case, under various additional assumptions on the characteristic, existence of liftings, etc., was proven in \cite{KP}. The main aim of this paper is to provide such a generalization in all dimensions without any assumptions,
answering the conjecture posed in \cite[Remark 4.11]{KP}.

\medskip

Let $X$ be a smooth projective variety of dimension $d\ge 1$ defined over an algebraically closed field $k$ (of arbitrary characteristic).  A vector bundle $E$ on $X$ is called \emph{numerically flat} if both $E$ and $E^*$ are nef. We prove the following theorem:

\begin{theorem}\label{main-theorem}
Let $E$ be a numerically flat vector bundle of rank $r$. Then for 
any ample line bundle $A$ on $X$ and any	
$1\le s<r$ there exists a sequence of finite surjective morphisms $f_m:X_m\to X$ from smooth projective varieties, 
a sequence of filtrations of $f_m^*E$ by subbundles $$S_{0,m}=0\subset S_{1,m}\subset...\subset S_{r-1,m}\subset S_{r,m}=f_m^*E,$$  
and some rational numbers $\gamma_1,...,\gamma_r$ such  that for $i=1,...,r$ the quotients $S_{i,m}/S_{i-1,m}$ are line bundles and we have 
$$\frac{(f_m)_*c_1(S_{i,m}/S_{i-1,m})}{\deg f_m} =\frac{\gamma_i}{m} \cdot c_1(A)$$
in the rational Chow group $A^1(X)\otimes \QQ$. 	
\end{theorem}

This theorem  generalizes (and strengthens) the results of Parameswaran and Subramanian \cite{PS} from curves to higher-dimensional varieties.

In fact, we prove a more precise version of this theorem, taking into account the behaviour of 
the S-fundamental group scheme (see Corollary \ref{cor-flag}), which contains information about all numerically flat vector bundles on a given variety. Along the way, we also prove  
some results related to the homotopy exact sequence for the S-fundamental group scheme (see Section 2).

The main new idea in the proof of the above theorem is a different choice of a general complete 
intersection in the flag or Grassmann scheme, along with the use of some push-forward formulas for flag bundles of Darondeau and Pragacz (see \cite{DP}).

\medskip

The structure of the paper is as follows. Section 1 contains some preliminaries. In Section 2 we prove several new results on the S-fundamental group scheme. Section 3 contains an easy version of Theorem 
\ref{main-theorem} that deals with subbundles of a given rank (this is a direct generalization of 
\cite[Theorem 4.1]{PS} to higher dimensions). Section 4 contains the proof of Theorem \ref{main-theorem}.
In Section 5 we prove a strong restriction theorem for bundles that can be approximated as in the statement of Theorem \ref{main-theorem}. This generalizes \cite[Theorem 3.4]{KP} that deals with the rank $2$ case in positive characteristic.

\subsection*{Notation}

A vector bundle $E$ is a locally free sheaf of finite rank. A subbundle $F\subset E$ is a locally free sheaf such that the quotient $E/F$ is also locally free.

If $X$ is a projective variety then we write $\Pic X$ for the Picard group of $X$ and $A^1 (X)$ for the group of codimension $1$ cycles on $X$ modulo algebraic equivalence. If $X$ is smooth the first Chern class $c_1: \Pic X\to A^1(X)$ defines an isomorphism of groups.  A class $\alpha \in A^1(X)$ is \emph{numericallly trivial} if $\int_X \alpha \cdot [C]=0$ for every projective curve $C\subset X$. In this case we write $\alpha \equiv 0$.
Let us recall that  a class $\alpha \in A^1(X)$  is numerically trivial if and only if the corresponding class  in $A^1(X)\otimes \QQ$ vanishes.

\section{Preliminaries}

\subsection{Numerically flat vector bundles}

Let $X$ be a smooth projective variety of dimension $d\ge 1$
defined over an algebraically closed field $k$.

Let us now fix an ample divisor $H$.  We say that $E$ is \emph{strongly slope $H$-semistable} if either $k$ has characteristic $0$ and $E$ is slope $H$-semistable, or $k$ has positive characteristic and all Frobenius pull backs $(F_X^n)^*E$ are slope $H$-semistable.

Let us recall the following characterization of numerically flat vector bundles:

\begin{theorem}\label{num-flat}
	Let $H$ be a fixed ample divisor on $X$ and let $E$ be a vector bundle on $X$. Then the following conditions are equivalent:
	\begin{enumerate}
		\item $E$ is numerically flat,
		\item $E$ is strongly slope $H$-semistable and $\int_X (\ch  _1(E) \cdot H^{d-1})=\int_X (\ch  _2(E) \cdot H^{d-2})=0$,
		\item $E$ is nef with $c_1(E)\equiv 0$.
	\end{enumerate}
\end{theorem}

This follows from  \cite{Na} in the characteristic zero case and 
\cite[Proposition 5.1]{La1} in the positive characteristic case. See also \cite{FL} for a new algebraic proof of this theorem that is independent of the characteristic.

\subsection{Ampleness for nef vector bundles}

\medskip

The following lemma follows from the proof of \cite[Proposition 2.9]{Vi}. We recall its proof for the convenience of the reader.

\begin{lemma}\label{Viehweg}
Let $X$ be a projective scheme and let $E$ be a nef vector bundle on $X$. Let $\pi: \PP(E)\to X$ denote the projectivisation of $E$.
	Then for any ample line bundle $A$ on $X$ and  any positive integer $m$, the line bundle $L=\cO_{\PP (E)}(m)\otimes \pi^*A$ is ample. 
\end{lemma}

\begin{proof}
	By Seshadri's criterion of ampleness (see, e.g., \cite[Theorem 1.4.13]{Laz}), it is sufficient to prove that there exists $\epsilon >0$ such that for any irreducible curve $C$ on $\PP(E)$
	and every point $x\in C$ we have $\deg L_C>\epsilon \cdot \mult _xC$.
	
	If $\pi(C)$ is a point then $C$ is contained in $\pi^{-1}(\pi (C))\simeq \PP^{r-1}$. Let $d(C)$ denote the degree of $C$ in $\PP^{r-1}$. Then 
	  $$\deg L_C=\deg \cO_{\PP (E)}(m)_C=m\cdot d(C) \ge d(C)\ge {\mult} _xC.$$
	
Again using Seshadri's criterion of ampleness we know that there exists  $\epsilon '>0$ such that for any irreducible curve $C'$ on $X$ we have $\deg A_{C'}>\epsilon '\cdot \mult _xC'$ for all $x\in X$.
If $\pi (C)$ is not contracted to a point then its image is an irreducible curve $C'$. Then for every point $x\in \PP(E)$ we have
  $$\deg L_C\ge \deg (\pi^*A)_C=\deg A_{C'}>\epsilon '\cdot \mult _{\pi (x)}C'\ge \epsilon ' \cdot \mult _x C. $$
  So we can take $\epsilon =\min (\epsilon ', 1)$.
\end{proof}

\subsection{Universal push-forward formulas for full flag varieties}\label{Subsection-DP}

Let $E$ be a rank $r$ vector bundle on a smooth variety $X$ defined over an algebraically closed field.

Let $\pi: \Flag (E)\to X$ denote the full flag bundle of $E$ of subbundles of rank $1,...,r-1$ and let 
$S_0=0\subset S_1 \subset S_2\subset...\subset S_{r-1}\subset \pi^*E$ be the universal flag on $\Flag (E)$.
For $i=1,..., r-1$ we set  $\xi _i=-c_1 (S_{r-i}/S_{r-1-i})$.

For a monomial $m$ and a polynomial $f$ we denote by $[m](f)$ the coefficient of $m$ in $f$.
We will need the following special case of the  Darondeau--Pragacz formula (see \cite[Theorem 1.1]{DP}):

\begin{theorem}\label{DP-formula}
If $c_1 (E)=0$ in $A^1(X)$ then for any homogeneous polynomial $f\in k[t_1,...,t_{r-1}]$ 
we have
$$\pi_*f(\xi_1,...,\xi_{r-1})=\left\{\begin{array}{ll}
		\left[\prod _{i=1}^{r-1}t_i^{r-1}\right]\left(  f(t_1,...,t_{r-1})\prod _{1\le i<j\le r-1} (t_i-t_j)\right)[X]&\hbox{if $f$ has degree $\binom {r}{2}$,} \\
	0&\hbox{if $f$ has degree $\binom{r}{2}+1$.}\\
	\end{array}\right.
$$
\end{theorem}

\begin{proof}
Let us recall that $\Flag (E)$  has relative dimension $\binom {r} {2}$ over $X$. So \cite[Theorem 1.1]{DP} allows us to compute $\pi_*f(\xi_1,...,\xi_{r-1})$ as the above multiple of $[X]$ if $f$ has degree $\binom {r}{2}$ and as a certain multiple of $c_1(E)$ if $f$ has degree $\binom{r}{2}+1$. This last class vanishes by our assumption.
\end{proof}

\section{Some properties of the S-fundamental group scheme}

Let $X$ be a connected reduced proper scheme over an algebraically closed field $k$
and let $x$ be a fixed $k$-point of $X$. Let $\cC^{\nf} (X)$ denote the category of numerically flat vector bundles on $X$ with the fiber functor $\omega_x: \cC^{\nf} (X)\to \Vect _k$
given by sending $E$ to the fiber  $E(x)$.
Together with tensor product and the unit $\cO_X$ this forms a neutral Tannakian category.
Then one defines the \emph{S-fundamental group scheme} $\pi_1^S(X,x)$ with base point $x$ as the affine $k$-group scheme which is Tannaka dual to this category (cf. \cite[Definition 2.1]{La2}).

With the above notation we have the following lemma.

\begin{lemma}\label{injectivity} 
For any numerically flat vector bundle $E$ on $X$ the canonical morphism $\Gamma (X, E)\otimes _k \cO_X\to E$ is injective and its image is the maximal trivial subbundle of $E$.
\end{lemma}

\begin{proof}
Let $K$ be the kernel of the evaluation map $\Gamma (X, E)\otimes _k \cO_X\to E$.
By construction we have $\Gamma (X, K)=0$. Since  the category $\cC^{\nf}(X)$ is abelian, $K$ is a numerically flat vector bundle.

We claim that any numerically flat subbundle of the trivial bundle 
with no non-zero sections is zero. Otherwise, we choose a bundle $K'$ of the smallest positive rank with these properties. There exists a non-zero map $K'\to \cO_X$. Since the kernel of this map is still a numerically flat subbundle of the trivial bundle, it is $0$ by our assumption on the rank. But then $K'\to \cO_X$ is an isomorphism, contradicting our assumption $\Gamma (X, K')=0$. This shows that $K=0$. 

Since sections of every subsheaf of $E$ are contained in 
$\Gamma (X, E)$ and $\Gamma (X, E)\otimes _k \cO_X\to E$ is injective, the maximality of the image is clear. 
\end{proof}

\medskip

Let us recall that a morphism of schemes $f: Y\to X$ is called \emph{separable} if it is flat and for every $x\in X$ the fiber $Y_x$ is geometrically reduced over $\kappa (x)$ (see \cite[Expos\'e X, D\'efinition 1.1]{SGA}).
Let us warn the reader that this notion is often used to talk about dominant morphisms of $k$-varieties  $f: Y\to X$ 
for which the corresponding field extension $k(X)\to k(Y)$ is separably generated. This is equivalent to saying that $f$ is separable (in SGA1 sense) over the generic point of $X$ and in this case we say that $f$ is \emph{generically separated}. 
Note that any dominant morphism of $k$-varieties is generically separated if $k$ has characteristic zero.  On the other hand, a finite morphism of smooth projective curves is separable if and only if it is \'etale.

\medskip

Let us recall the following lemma (see \cite[Lemma 8.1]{La1} ).

\begin{lemma}\label{Lemma-8.1}
Let $f: Y\to X$ be a surjective separable morphism of proper varieties over an algebraically closed field $k$.  If $f_*\cO_Y=\cO_X$ then for any $y\in Y(k)$ the natural map $\pi_1^S(Y,y)\to \pi_1^S(X, f (y))$ is faithfully flat.
\end{lemma}

Note that the formulation of  \cite[Lemma 8.1]{La1} does not contain the assumption that fibers are geometrically reduced but it is used in the proof of this lemma. 

The main aim of this section is to prove the following result.

\begin{theorem}\label{homotopy-exact-sequence}
Let $f: Y\to X$ be a surjective separable morphism of proper varieties over an algebraically closed field $k$. Assume that $f_*\cO_Y=\cO_X$ and 
for a general $k$-point $x$ of $X$ we have $\pi_1^S(Y_x, y)=0$ for some $k$-point $y$ of the fiber $Y_x$.
Then the natural map $\pi_1^S(Y,y)\to \pi_1^S(X, f (y))$ is an isomorphism.
\end{theorem}

\begin{proof}
By Lemma \ref{Lemma-8.1} we know that  $\pi_1^S(Y,y)\to \pi_1^S(X, f (y))$ is faithfully flat. So we need only to show that it is also a closed immersion.
	
Let $E$ be a numerically flat vector bundle of rank $r$ on $Y$. Our assumptions imply that the fiber $Y_x$ over any  $x\in X(k)$ is reduced (as $f$ is separable) and connected (since $f_*\cO_Y=\cO_X$). 
The restriction $E_x$ to any fiber $Y_x$  is numerically flat, so Lemma \ref{injectivity} implies that $h^0(Y_x, E_x)\le r$ for all $x\in X(k)$. So by semicontinuity of cohomology, $h^0(Y_x, E_x)=r$ for all $x\in X(k)$. Hence by Grauert's theorem $f_*E$ is locally free of rank $r$. 

Lemma \ref{injectivity} implies also that for all $x\in X(k)$ the evaluation maps
$\Gamma (Y_x, E_x)\otimes _k \cO_{Y_x}\to E_x$ are isomorphisms. 
Therefore the canonical map
$$f^*f_*E\to E$$ 
is also an isomorphism. Note that if $f^*F$ is nef for some vector bundle $F$ on $X$ then $F$ is also nef. So the above isomorphism implies that
$f_*E$ is numerically flat.
Now \cite[Proposition 2.21 (b)]{DM} implies that the
natural homomorphism $ \pi_1^S( Y, y) \to \pi_1^S(X,f(y))$ is a closed immersion and
therefore it is an isomorphism.
\end{proof}

\begin{remark}
Note that the above theorem is a weak version of the homotopy exact sequence for proper separable morphisms with geometrically connected fibers (see \cite[Expos\'e X, Corollaire 1.4]{SGA}).
It is known that in general exactness of this sequence fails even for Nori's fundamental group scheme and  a smooth projective morphism between smooth projecive varieties 
(see \cite[Section 2]{EPV}; note that this part of the paper is correct). So exactness of the homotopy exact sequence
also fails in general for the S-fundamental group scheme. 
\end{remark}

\medskip

We will also need  the following easy lemma.

\begin{lemma}\label{birational}
Let $f: Y\to X$ be a morphism of connected reduced proper $k$-schemes. If $f_*\cO_Y=\cO_X$ 
and $\pi_1^S(Y,y)=0$  for some $y\in Y(k)$ then $ \pi_1^S(X, x)=0$ for every $x\in X(k)$. 
\end{lemma}

\begin{proof}
	If $F$ is a numerically flat bundle on $X$ then
	$f^*F$ is trivial, so $F\simeq f_*f^*F$ is also trivial. This implies that  $ \pi_1^S(X, x)=0$.
\end{proof}

The following corollary generalizes \cite[Proposition 8.2]{La1} (but with a completely different proof).

\begin{proposition}\label{homogeneous}
	Let $G$ be a connected reductive algebraic group over an algebraically closed field $k$
	and let $P\subset G$ be a parabolic subgroup.
	Then  for any Schubert variety $X$  in $G/P$ and any  $x\in X(k)$ we have
	 $\pi_1^S(X, x)=0$. In particular, we have  $\pi_1^S(G/P, x)=0$ for any $ x\in (G/P) (k)$.
\end{proposition}

\begin{proof}
By assumption $X$ is the closure of $BwP/P$, where $B\subset P$ is a Borel subgroup and $w$
is an element of the Weyl group $W$. The canonical map $\pi: G/B\to G/P$ is a locally trivial fibration with fiber $P/B$ and $X$ is the image of $Y=\pi^{-1}(X)$, which is a Schubert variety in $G/B$. Since Shubert varieties are normal (see \cite[Part II, 
Proposition 14.15]{Ja}), Theorem \ref{homotopy-exact-sequence} reduces the problem to Shubert varieties in the full flag variety $G/B$. So from now on we assume that $P=B$
is a Borel subgroup. Now let us recall that any Schubert variety $Y$ in $G/B$
admits a proper birational morphism $\tilde Y\to Y$ from the so called Bott--Samelson--Demazure--Hansen variety $\tilde Y$
(see \cite[Part II, Chapter 13]{Ja}). There also exists a sequence 
$\tilde Y=Y_m\to Y_{m-1}\to ...\to Y_1\to \Spec k$ of locally trivial fibrations with fibre $\PP^1$. So Theorem \ref{homotopy-exact-sequence} implies that $\pi_1^S(\tilde Y, \tilde y)=0$ for any $\tilde y\in \tilde Y (k)$. Then Lemma \ref{birational} implies that 
 $\pi_1^S(X, x)=0$  for any $x\in X(k)$. 
\end{proof}

\section{Subbundles via Grassmann schemes}

\begin{theorem}\label{main-limit}
	Let $X$ be a smooth projective variety of dimension $d\ge 1$ defined over an algebraically closed field $k$. Let $E$ be a rank $r$ nef vector bundle on $X$ with numerically trivial $c_1(E)$. Then for any ample line bundle $A$ on $X$ and any $1\le s<r$
	there exists a sequence of finite generically separable surjective morphisms $f_m:X_m\to X$ from smooth projective varieties and a sequence of rank $s$ subbundles $S_m\subset f_m^*E$ such that
	for some sequence $(a_m)$ of rational numbers converging to $0$, we have
		$$\frac{(f_m)_*c_1(S_m)}{\deg f_m} =a_m\cdot c_1(A)$$
	in $A^1(X)\otimes \QQ$.
\end{theorem}

\begin{proof}
	Let $\pi :\Gr (s, E)=\Gr (E, r-s)\to X$ denote the Grassmann  scheme of rank $s$ subbundles of $E$
	(or, equivalently,  rank $r-s$ quotients of $E$).
	This is an $X$-scheme representing the functor $\Sch /k\to \Sets$ sending a scheme $T$ to the set of all constant rank $s$ subsheaves $G\subset E_T$ with locally free quotients  $E_T/G$.
	This scheme comes with the canonical Pl\"ucker embedding $\Gras (s, E)\to \Gras (1, \bigwedge ^s E)=\PP(\bigwedge ^s E^*)$ (of schemes over $X$),
	given by sending $G\subset E_T$ to the quotient $\bigwedge ^s E^*_T\to (\det G)^*$.
	
By Proposition \ref{num-flat}  $E$ is numerically flat and hence $\bigwedge ^s E^*$ is also numerically flat. Therefore
	$\cO_{\PP(\bigwedge ^s E^*)}(1)$ is nef. 
	
	Let $A$ be an ample line bundle on $X$ and let $m$ be a positive integer. By Lemma \ref{Viehweg} the line bundle $\cO_{\Gras (s,E)}(m)\otimes \pi^*A$ is the restriction of an ample line bundle on $\PP(\bigwedge ^s E^*)$, so it is also ample. Therefore for all large $n$ the line bundle $L_{m,n}:=\cO_{\Gras (s,E)}(mn)\otimes \pi^*A^{\otimes n}$ is very ample. 
	
For every positive integer $m$  we choose some $n_m$ such that $L_{m,n_m}$ is very ample and we denote by 
 $X_m\subset \Gr (s, E)$ a general complete intersection of $s(r-s)$ elements of the linear system $|L_{m,n_m}|$. Let $f_m: X_m \to X$ be the restriction of $\pi$ to $X_m$.
	Restricting the universal exact sequence 
	$$0\to S\to \pi^*E\to Q\to 0$$
	to $Y$, we get a short exact sequence
	$$0\to S_{X_m}\to f_m^*E\to Q_{X_m}\to 0.$$
	
	Let us set $\xi=  c_1 (\cO_{\Gras (s,E)}(1))$. Then we have
	$$\deg f_m\cdot [X]= (f_m)_*[X_m]= \pi_*c_1(L_{m,n})^{s(r-s)}= (mn)^{s(r-s)}\pi_*\xi ^{s(r-s)}.$$
Since $c_1(Q)=\xi$, we also have
	\begin{align*}
	(f_{m})_*c_1(Q_{X_m}) &=\pi_* (c_1 (Q)\cdot [X_m])=	\pi_* (L_{m,n}^{s(r-s)}\cdot \xi )
	=n^{s(r-s)}\pi_* \left( (m\, \xi +\pi^*c_1(A))^{s(r-s)}\cdot \xi  \right) \\
		&= n^{s(r-s)}\left(
		m^{s(r-s)} \pi_*\xi^{s(r-s)+1}+
		m^{s(r-s)-1}s(r-s)\,  \pi_*( \xi ^{s(r-s)}\pi^*c_1(A))\right),\\
	\end{align*}
where $\pi_* \xi ^i=0$ for $i<{s(r-s)}$ for dimensional reasons.

Now the J\'ozefiak--Lascoux--Pragacz formula (see \cite[Proposition 1]{JLP}) allows us to compute $ \pi_* \xi^{s(r-s)+1} $ as a certain multiple of $c_1(E)$ (see \cite[Theorem 0.1]{KT} for an explicit formula).  Hence our assumption on $c_1(E)$ implies that
	\begin{align*}
		(f_{m})_*c_1(Q_{X_m}) &=n^{s(r-s)}
		m^{s(r-s)-1}s(r-s)\, \pi_* \left(  \xi ^{s(r-s)}\pi^*c_1(A) \right)
		\\
		&=n^{s(r-s)} m^{s(r-s)-1}s(r-s)\, \left(\pi_* \xi ^{s(r-s)}\right)
		c_1(A)\\
	\end{align*}
in $A^1(X)\otimes \QQ$.	It follows that
	$$\frac{(f_m)_*c_1(Q_{X_m}) }{\deg f_m}= \frac{s(r-s)\,  c_1(A)}{m}.$$
	Since $(f_m)_*c_1(S_{X_m})=-(f_m)_*c_1(Q_{X_m})$, we get the required assertion.
\end{proof}

\begin{corollary}\label{cor-subbundle}
	Let $H$ be ample divisor on $X$ and let $E$ be a strongly $H$-semistable vector bundle of rank $r$. If $\int_X\Delta (E)H^{d-2}=0$ then for any ample line bundle $A$ on $X$ and any	
	$1\le s<r$ there exists a sequence of finite surjective morphisms $f_m:X_m\to X$ from smooth projective varieties and a sequence of rank $s$ locally free subsheaves $S_m\subset f_m^*E$ such that
	$$\frac{(f_m)_*c_1(S_m)}{\deg f_m} =\frac{s}{r}c_1(E)+a_m c_1(A)$$
	in $A^1(X)\otimes \QQ$, where $a_m$ is some sequence of rational numbers converging to $0$.	
	In particular, we have
	$$\lim _{m\to \infty}\frac{\mu_{f_m^*H}(S_m)}{\deg f_m}=\mu_H(E).$$
\end{corollary}

\begin{proof}
	We can use the Bloch--Gieseker branched covering trick to reduce to the previous theorem.
	Let us set $L=\det E$. By \cite[Lemma 2.1]{BG} there exists a smooth projective variety $Y$ and
	a finite surjective morphism $\tilde f: \tilde X\to X$ such that $\tilde f^*L\simeq M^{\otimes r}$ for some line bundle $M$ on $Y$. Then $\tilde E=\tilde f^*E \otimes M^{-1}$ is strongly $\tilde f^*H$-semistable with  
	$$\int_{\tilde X}\Delta (\tilde E)(f^*H)^{d-2}=\int_{\tilde X}\Delta (f^*E)(f^*H)^{d-2}=\deg f\cdot \int_X\Delta (E)H^{d-2}=0$$
	Since $\det \tilde E=\cO_{\tilde X}$, Theorem \ref{num-flat} implies that $\tilde E$ is numerically flat. 
	
	Now for any generically finite morphism $f: Y\to \tilde X$ and any subsheaf $S\subset f^*\tilde E$  
	with $f_*c_1(S)=a\deg f \cdot c_1(\tilde f^*A)$
	we have
	$$ \frac{(\tilde f\circ f)_*c_1(S\otimes f^* M)}{\deg (\tilde f\circ f)}=
\frac{\tilde f_*\left(	f_*c_1(S)+s\deg f\cdot c_1(M)\right)}{\deg (\tilde f\circ f)}=
	\frac{\tilde f_*\left( ac_1(\tilde f ^*A) +\frac{s}{r} c_1(\tilde f^*E)\right)}{\deg \tilde f}=ac_1(A)+\frac{s}{r} c_1(E),
	$$

Since
	$S\otimes f^*M\subset (\tilde f\circ f)^*E$, the first part of the corollary follows from Theorem \ref{main-limit}.
	The second part follows from the first one by intersecting the obtained equality with $H^{d-1}$.
\end{proof}

\section{Flags of line bundles via flag schemes}

\begin{theorem}\label{main-flag}
	Let $X$ be a smooth projective variety of dimension $d\ge 1$ defined over an algebraically closed field $k$. Let $E$ be a rank $r$ nef vector bundle on $X$ with numerically trivial $c_1(E)$. Then for any ample line bundle $A$ on $X$
	there exists a sequence $\{f_m\}_{m\in \NN}$ of finite generically separable surjective morphisms $f_m:X_m\to X$ from smooth projective varieties,  a sequence $\{S_{\bullet , m}\}_{m\in \NN}$ 
	of filtrations of $f_m^*E$ by subbundles $$S_{0,m}=0\subset S_{1,m}\subset...\subset S_{r-1,m}\subset S_{r,m}=f_m^*E$$  
of increasing rank, and some rational numbers $\gamma_1,...,\gamma_r$, such  that for $i=1,...,r$ we have 
	$$\frac{(f_m)_*c_1(S_{i,m}/S_{i-1,m})}{\deg f_m} =\frac{\gamma_i}{m}\cdot c_1(A)$$
	in $A^1(X)\otimes \QQ$. Moreover, we can choose $f_m:X_m\to X$ so that for any  $x_m\in X_m (k)$ the induced map of $S$-fundamental group schemes $\pi_1^S(X_m, x_m)\to \pi_1^S(X, f_m(x_m))$
	is faithfully flat. If $\dim X\ge 2$ then we can find $f_m$ for which $\pi_1^S(X_m, x_m)\to \pi_1^S(X, f_m(x_m))$ is an isomorphism.
\end{theorem}

\begin{proof}
Let $V$ be an $r$-dimensional vector space over $k$ and let us fix a full flag
$$ V=V_r\twoheadrightarrow V_{r-1}\twoheadrightarrow ... \twoheadrightarrow V_2\twoheadrightarrow V_{1} ,$$
where  $\dim _{k}V_j=j$. We have  a standard  Borel subgroup  $B\subset \GL (V)$ corresponding to linear
maps preserving this flag.  Note that  we have the canonical closed embedding
$$\GL (V)/B\hookrightarrow  \Gras (V,1 )\times _k\Gras (V,2)\times_k ... \times_k \Gras (V, r-1) $$
obtained on the level of functors of points $(\Sch/k)^{\rm op}\to \Sets$
by sending an $S$-point of $\GL (V)/B$ corresponding to a full flag
$$ V_{S}\twoheadrightarrow E_{r-1}\twoheadrightarrow ... \twoheadrightarrow E_2\twoheadrightarrow E_{1} $$
to a tuple  $(V_S\twoheadrightarrow E_{j})_{j=1,...,r-1}$.

Let  $P\to X$ be the frame bundle of $E$. This is a principal $\GL (V)$-bundle such that $E$ is isomorphic to the fibration $P\times _{\GL (V)}V\to X$ associated to $P$ through the action of $\GL (V)$ on $V$. 

Let $\pi : \Flag (E)\to X$ denote the full flag bundle of $E$. It can be constructed as a fibration $P\times _{\GL (V)}\GL (V)/B \to X$ associated to  $P$ through the action of $\GL (V)$ on $\GL (V)/B$.
The above embedding of $\GL (V)/B$ into a product of Grassmannians induces a closed embedding of $X$-schemes
$$\Flag (E)\hookrightarrow  \Gras (E,1 )\times _X\Gras (E,2)\times_X ... \times_X \Gras (E, r-1) .$$
We have a universal quotient flag $ \pi^*E \twoheadrightarrow Q_{r-1}\twoheadrightarrow ... \twoheadrightarrow Q_{2}\twoheadrightarrow Q_{1} $
and the corresponding universal flag of subbbundles $S_1\subset S_2\subset...\subset S_{r-1}\subset \pi^*E$ that are related by  $Q_i=\pi^*E/S_{r-i}$. 
Let $p_s: \Flag (E)\to  \Gras (E,s)$ be the canonical map obtained by composing the above embedding with the projection. Let us set $\tau_s= c_1 (p_s^*\cO_{\Gras (E,s)}(1) )$.
Then in the notation of Subsection \ref{Subsection-DP} we have 
$c_1(Q_s)=\tau_s=\sum _{i=1}^{r-s} \xi_{r-i}$ in $A^1(X)\otimes \QQ$ for $s=1,...,r-1$.
Hence  $\xi _s=\tau_{s}-\tau _{s+1}$ for $s=1,..., r-1$ (here we put $\tau_r=0$).

	Let $A$ be an ample line bundle on $X$ and let $\underline {m}=(m_1,...,m_{r-1})$ be an $(r-1)$-tuple of positive integers. By Lemma \ref{Viehweg}  for every $1\le i\le r-1$ the line bundle $\cO_{\Gras (i,E)}(m_i)\otimes \pi^*A$ is the restriction of an ample line bundle on $\PP(\bigwedge ^i E^*)$, so it is also ample.
	It follows that the line bundle
	$$\boxtimes _{i=1}^{r-1}\cO_{\Gras (i,E)}(m_i)\otimes \pi^*A$$ and  its restriction $L_{\underline{m}}$ to $\Flag (E)$
	are also ample.
So  for all large $n$ the line bundle $L_{\underline {m},n}:=L_{\underline{m}}^{\otimes n}$ is very ample.

Proposition \ref{homogeneous} implies that the fibers of $\pi: \Flag (E)\to X$ have trivial $S$-fundamental group schemes, so by Theorem \ref{homotopy-exact-sequence} 
for any $k$-point $x\in (\Flag (E))(k)$ the map	
$\pi_1^S(\Flag (E), x)\to \pi_1^S(X, \pi (x))$ is an isomorphism.

Let $X_{\underline{m}}\subset \Flag (E)$ be a general complete intersection of $\binom{r}{2}$ elements of the linear system $|L_{
\underline{m},n}|$. Let $f_{\underline{m}}: X_{\underline{m}}\to X$ be the restriction of $\pi$ to $X_{\underline{m}}$.
By Bertini's theorem $X_{\underline{m}}$ is a smooth projective variety. By \cite[Theorem 2.6]{Hi}  $f_{\underline{m}}$ is a finite morphism. It is generically separable by \cite[Proposition 3.3]{Hi}.
\cite[Theorem 10.2]{La1} implies that
for any $x\in X_{\underline{m}}(k)$
 the map  $\pi_1 ^S (X_{\underline{m}}, x)\to \pi_1^S(\Flag (E), x)$
is faithfully flat, so $\pi_1^S(X_{\underline{m}}, x)\to \pi_1^S(X, f_{\underline{m}}(x))$
is also faithfully flat.  If moreover $\dim X \ge 2$ then \cite[Theorem 10.4]{La1} implies that the map 
$\pi_1 ^S (X_{\underline{m}}, x)\to \pi_1^S(\Flag (E), x)$
is an isomorphism so $\pi_1^S(X_{\underline{m}}, x)\to \pi_1^S(X, f_{\underline{m}}(x))$
is also an isomorphism.

Restricting the universal flag to  $X_{\underline{m}}$, we obtain a flag of bundles on $X_{\underline{m}}$.	
Then
	$$\deg f_{\underline{m}}\cdot [X]= (f_{\underline{m}})_*[X_{\underline{m}}]= \pi_*c_1(L_{\underline {m},n})^{\binom{r}{2}}= n^{\binom{r}{2}}\pi_*
	\left(\sum m_i\tau_i +(r-1)\pi^*c_1(A)\right)^{\binom{r}{2}}=n^{\binom{r}{2}}\pi_*
	\left(\sum m_i\tau_i \right)^{\binom{r}{2}},$$
	where the last equality holds by the projection formula and  dimensional reasons.
We also have
	\begin{align*}
	(f_{\underline{m}})_* (c_1 (Q_s|_{X_{\underline{m}}}))&= \pi_* (c_1 (Q_s)\cdot [X_{\underline{m}}]) =	\pi_* (L_{{\underline{m}},n}^{\binom{r}{2}}\cdot \tau_s)
	=n^{\binom{r}{2}}\pi_* \left( \left(\sum m_i\tau_i +(r-1)\pi^*c_1(A)\right)^{\binom{r}{2}}\cdot \tau_s \right) \\
		&= n^{\binom{r}{2}}\left(
	 \pi_*\left(\left(\sum m_i\tau_i \right)^{\binom{r}{2}}\cdot \tau _s\right)+ (r-1)\binom{r}{2}\pi_*\left(\left(\sum m_i\tau_i \right)^{\binom{r}{2}-1}\cdot \tau _s\right)\cdot
		c_1(A) \right),\\
	\end{align*}
since the remaining elements vanish for dimensional reasons.

Now Theorem \ref{DP-formula} allows us to compute $ \pi_*\left(\left(\sum m_i\tau_i \right)^{\binom{r}{2}}\cdot \tau _s\right)$ as a certain multiple of $c_1(E)$.  Hence our assumption on $c_1(E)$ implies that
$$
	(f_{\underline{m}})_* (c_1 (Q_s|_{X_{\underline{m}}}))= n^{\binom{r}{2}} (r-1)\binom{r}{2}\pi_*\left(\left(\sum m_i\tau_i \right)^{\binom{r}{2}-1}\cdot \tau _s\right)\cdot	 c_1(A) .	
$$
in $A^1(X)\otimes \QQ$.	Note that 
$$\sum_{i=1}^{r-1} m_i\tau_i=\sum_{i=1}^{r-1} \left(\sum _{j=1}^{i}m_j\right) \xi _i.$$

Now for any positive integer $m$ we consider the $(r-1)$-tuple $\underline{m}=(m,...,m)$.
For this tuple we write $X_m$ and $f_m$  instead of $X_{\underline{m}}$ and $f_{\underline{m}}$.
Let us set
$$
\alpha:=\left[\prod _{i=1}^{r-1}t_i^{r-1}\right]\left(  \left(\sum _{i=1}^{r-1}it_i\right)^{\binom{r}{2}}\prod _{1\le i<j\le r-1} (t_i-t_j)\right)$$
and
$$
\beta_s:=\left[\prod _{i=1}^{r-1}t_i^{r-1}\right]\left(  \left(\sum _{i=1}^{r-1}it_i\right)^{\binom{r}{2}-1}\left(\sum _{i=s}^{r-1}t_i\right)\prod _{1\le i<j\le r-1} (t_i-t_j)\right) .$$
Then Theorem \ref{DP-formula} implies that
$$\deg f_m=(nm)^{\binom{r}{2}}   \alpha $$
(so in particular $\alpha>0$),
and
$$
	(f_m)_* (c_1 (Q_s|_{X_m}))=  n^{\binom{r}{2}}  m^{\binom{r}{2}-1}(r-1)\binom{r}{2}  \beta_s  c_1(A) .	
$$
But then we get
$$\frac{(f_m)_*c_1(Q_{X_m, s}) }{\deg f_m}= \frac{(r-1)\binom{r}{2} \beta _s\,  c_1(A)}{m\alpha },$$
which easily implies the required assertion.
\end{proof}

In the same way as in proof of Corollary \ref{cor-subbundle}, the above theorem implies the following corollary:

\begin{corollary}\label{cor-flag}
Let $H$ be ample divisor on $X$ and let $E$ be a  strongly $H$-semistable vector bundle of rank $r$. If $\int_X\Delta (E)H^{d-2}=0$ then for any ample line bundle $A$ on $X$ there exists a sequence of finite surjective morphisms $f_m:X_m\to X$ from smooth projective varieties and 
a sequence of filtrations of $f_m^*E$ by subbundles $$S_{0,m}=0\subset S_{1,m}\subset...\subset S_{r-1,m}\subset S_{r,m}=f_m^*E$$  
so that for $i=1,...,r$ the quotients $S_{i,m}/S_{i-1,m}$ are line bundles and  we have 
$$\frac{(f_m)_*c_1(S_{i,m}/S_{i-1,m})}{\deg f_m} =\frac{c_1(E)}{r}+\frac{\gamma_i}{m} \cdot c_1(A)$$
in $A^1(X)\otimes \QQ$ for some rational numbers $\gamma_1,...,\gamma_r$. 
\end{corollary}

\medskip

\begin{remark}
One can easily generalize the above theorem to principal bundles. 
Namely, let  $G$ be a semisimple group over an algebraically closed field $k$ and let $P\to X$
be a principal $G$-bundle. Let us assume that a vector bundle $P(\mathfrak{g})\to X$ associated to $P\to X$ via the adjoint representation is nef (note that it is automatically of degree $0$, even if $G$ is just reductive).  Note that by the Ramanan--Ramanathan theorem (see \cite[Theorem 3.23 and p.~290]{RR})
and Theorem \ref{num-flat}, this is equivalent to strong $H$-semistability of $P$ and vanishing of $\int_X\Delta(P(\mathfrak{g}))\cdot H^{d-2}$.

Then one can find  a sequence of finite generically separable surjective morphisms $f_m:X_m\to X$ from smooth projective varieties and a sequence of reductions of structure groups of $f_m^*P$ to a fixed Borel 
subgroup $B\subset G$ such that for each fixed character of $B$ the sequence of associated line bundles has degrees that after dividing by the degrees of $f_m$ converge to $0$. 

One can also prove an analogous formulation for an arbitrary reductive group $G$. Both these facts follow immediately from Theorem \ref{main-flag}.
\end{remark}

\section{Strong restriction theorem}

One can also prove the following partial converse of Corollary \ref{cor-flag}. Here we need only behaviour of slopes of pull-backs of $E$.

\begin{proposition}\label{converse}	
Let $X$ be as in Theorem \ref{main-flag}.
	Let $H$ be a very ample divisor and let $E$ be a rank $r$ vector bundle on $X$. Assume that there exists a sequence of finite surjective morphisms $f_m:Y_m\to X$ from smooth projective varieties and 
	a sequence of filtrations of $f_m^*E$ by subbundles $$S_{0,m}=0\subset S_{1,m}\subset...\subset S_{r-1,m}\subset S_{r,m}=f_m^*E$$  
	so that for $i=1,...,r$ the quotients $S_{i,m}/S_{i-1,m}$ are line bundles and
	$$\lim _{m\to \infty } \frac{\max_i \{c_1(S_{i,m}/S_{i-1,m})f_m^*H^{d-1}\}}{\deg f_m} =\mu_H(E).$$ 	
Then for a general complete intersection curve $C\in |H|\cap ...\cap |H|$
the restriction $E_C$ is semistable. In particular, $E$ is slope $H$-semistable.
\end{proposition}

\begin{proof}
	Let us fix some $0<\epsilon< \frac{1}{r(r-1)}$. By assumption there exists a morphism $f: Y\to X$ such that 
	$$0\le \frac{ \max_ic_1(S_{i}/S_{i-1})f^*H^{d-1}}{\deg f}  -\mu_H(E)<\epsilon .$$ 
	Let  $C\in |H|\cap ...\cap |H|$ be a general complete intersection curve
	and let  $D=(f^{-1}(C))_{\redu}$. Then by \cite[Th\'eor\`eme 1.2]{Zh} 
	$D$ is geometrically unibranched. In particular, $D$ is irreducible and the normalization $\nu: \tilde D\to D$ is (universally) injective. Let $g: \tilde D\to C$ be the composition of $f$ and $\nu $. Note that $g^*E_C$ has a filtration
	$$\nu^*(S_{0})_D=0\subset \nu^*(S_{1})_D\subset...\subset \nu^*(S_{r-1})_D\subset \nu^*S_{r}=\nu^*((f^*E)_D)=g^*E_C.$$  
Assume that $g^*E_C$ is not semistable and let $G$ be the maximal destabilizing subbundle of $g^*E_C$.
Then there exists $1\le i_0\le r$ such that the map $G\to \nu^*(S_{i_0}/S_{i_0-1})_D$ is nontrivial. Since $G$ is semistable and $\nu^*(S_{i_0}/S_{i_0-1})_D$ is a line bundle, this implies that $\mu (G)\le \deg (\nu^*(S_{i_0}/S_{i_0-1})_D)$. Therefore
$$\mu (G)\le \max_i\{\deg (\nu^*(S_{i}/S_{i-1})_D)\}=\deg g \cdot  \frac{ \max_ic_1(S_{i}/S_{i-1})f^*H^{d-1}}{\deg f}< \deg g \cdot (\mu_H(E)+\epsilon).$$
If $E_C$ is not semistable then by definition $E_C$ contains $F$ with $\mu(F)>\mu(E_C)=\mu_H(E)$.
But then $g^*F\subset g^*E_C$, so by the choice of $G$ we have 
$$\mu(G)\ge \mu (g^*F)=\deg g\cdot \mu (F)\ge \deg g \cdot \left(\mu _H(E)+\frac{1}{r(r-1)}\right).$$
This contradicts our choice of $\epsilon$.
\end{proof}

The above proposition generalizes and strengthens \cite[Theorem 3.4]{KP} that deals with the rank two case in positive characteristic.

\medskip

However, the following example shows that existence of the sequences of morphisms and filtrations as in Proposition \ref{converse} does not imply that $\int _X \Delta (E)H^{d-2}=0$ (cf.~Corollary \ref{cor-flag} for a strong converse).

\begin{example}
	Let $X$ be a smooth projective surface and let $H$ be a very ample line bundle on $X$. Assume that there exists a line bundle $L$ on $X$ such that $L.H=0$ and $L^2\ne 0$ (by the Hodge index theorem, such line bundles exist on any surface for which the rank of the N\'eron-Severi group is larger than $1$). 
	Then $E=L\oplus L^{-1}$ satisfies assumptions of Proposition \ref{converse} (e.g., with $f_m=\id _X$). In this example $c_1(E)=0$ but $E$ is not numerically flat as $\int_X c_2(E)=-L^2\ne 0$. Note also that there exists a smooth projective curve $C$ on $X$ such that $E_C$ is not semistable (e.g., one can take as $C$ a general curve in the linear system $|mH+L|$ for sufficiently large $m$) so Proposition 
	\ref{converse} cannot be strenghtened to all curves $C\subset X$. However, it can easily be strenghtened to all general complete intersection curves $C\in |m_1H|\cap ...\cap |m_{d-1}H|$.
\end{example}

\begin{corollary}\label{converse2}	
	Let $H$ be a very ample divisor and let $E$ be a rank $r$ vector bundle on $X$ defined over a base field $k$ of positive characteristic. Assume that there exists a sequence of finite surjective morphisms $f_m:Y_m\to X$ from smooth projective varieties and 
	a sequence of filtrations of $f_m^*E$ by subbundles $$S_{0,m}=0\subset S_{1,m}\subset...\subset S_{r-1,m}\subset S_{r,m}=f_m^*E$$  
	so that for $i=1,...,r$ the quotients $S_{i,m}/S_{i-1,m}$ are line bundles and
	$$\lim _{m\to \infty } \frac{\max_i \{c_1(S_{i,m}/S_{i-1,m})f_m^*H^{d-1}\}}{\deg f_m} =\mu_H(E).$$ 	
	Then for a very general complete intersection curve $C\in |H|\cap ...\cap |H|$
	the restriction $E_C$ is strongly semistable. In particular, $E$ is strongly slope $H$-semistable.
\end{corollary}

\begin{proof}
	Our assumptions imply that for any non-negative integer $n$, the bundle $(F_X^n)^*E$ also has an analogous sequence of the same morphisms and filtrations obtained by pull-backs of filtrations of $E$. So the assertion follows from Proposition \ref{converse}.
\end{proof}

\section*{Acknowledgements}

The author was partially supported by Polish National Centre (NCN) contract number
2021/41/B/ST1/03741.

\end{document}